\documentclass[12pt]{amsart}
\usepackage{a4wide}
\usepackage{psfrag}
\usepackage{graphicx}
\usepackage{verbatim}
\usepackage{amssymb}
\usepackage{todonotes}
\usepackage[makeroom]{cancel}

\numberwithin{equation}{section}

\newtheorem{lemma}{Lemma}[section]

\newtheorem{thm}{Theorem}

\theoremstyle{definition}

\newtheorem{defi}[lemma]{Definition}
\theoremstyle{remark}
\newtheorem{rem}{Remark}

\newcommand{\T}{\mathbb{T}}
\newcommand{\la}{\lambda}

\newcommand{\vf}{\varphi}

\newcommand{\ve}{\varepsilon}

\newcommand{\Z}{\mathbb{Z}}

\newcommand{\R}{\mathbb{R}}

\title[Schr\"odinger cocycles over expanding circle endomorphisms]{Positive Lyapunov exponent for some Schr\"odinger cocycles over strongly expanding circle endomorphisms}

\author{Kristian Bjerkl\"ov}
\email{bjerklov@kth.se}
\address{Department of Mathematics, KTH Royal Institute of Technology, 100 44 Stockholm, Sweden}

\begin{document}

\begin{abstract}

We show that for a large class of potential functions and big coupling constant $\lambda$ the Schr\"odinger cocycle over the expanding map
$x\mapsto bx ~( \text{mod} 1)$ on $\T$ has a Lyapunov exponent $>(\log\lambda)/4$ for all energies, provided that the integer $b\geq \la^3$.   
\end{abstract}

\maketitle
\section{Introduction}
Let $\T=\R/\mathbb{Z}$ and let $T:\T\to\T$ be the expanding map $T(x)=bx ~( \text{mod} 1)$, where $b\geq 2$ is an integer.
In this note we consider the Schr\"odinger cocycle on $\T\times \R^2$ defined by
$$
F_E:(x,y)\mapsto (T(x),A_E(x)y)
$$
where  
$$
A_E(x)=\left(\begin{matrix} 0 & 1 \\ -1 & \lambda v(x)-E
\end{matrix}
\right)\in SL(2,\R)
$$
and $v:\T\to\R$ is a continuous function, $\lambda\in\mathbb{R}$ is a coupling constant and $E\in\R$ is the energy parameter.

We let\footnote{We are interested in the time-evolution of $F_E$ for fixed $\la, v$ and $b$. 
Therefore we only indicate the dependence on $E$ and $n$.} 
$$
A^n_E(x)= A_E(T^{n-1}(x))\cdots A_E(x), ~ n\geq 1,
$$
and define the (maximal) Lyapunov exponent by
$$
L(E)=\lim_{n\to\infty}\frac{1}{n}\int_\T\log\|A_E^n(x)\|dx ~(\geq 0).
$$
Recall that the Lebesgue measure on $\T$ is an invariant measure for $T$. Since $T$ is ergodic with respect to this measure we have
\begin{equation}\label{pointwise}
\lim_{n\to\infty}\frac{1}{n}\log\|A_E^n(x)\|=L(E) \text{ for a.e. } x\in\T.
\end{equation}

For the important connection to the discrete Schr\"odinger operator we refer to the articles \cite{cs,d,dk,s} and references therein.

A natural question is to ask under which conditions (on $\la,v,b$ and $E$) we have $L(E)>0$. Here we are especially interested in conditions on $\la,v$ and $b$ which guarantees
$L(E)>\text{const.}>0$ for all $E\in\R$. Besides the problem in itself, 
which has a general interest in the theory of non-uniformly hyperbolic dynamical systems, such uniform lower bounds are many times important for deriving finer properties of
the associated Schr\"odinger operator (see, e.g., \cite{bs}).

Next follows a brief summary of previous results. It should be stressed that all the results hold for any $b\geq 2$.

In \cite{dk} it is shown that if $v$ is measurable, bounded and non-constant, and $\la>0$, then
$L(E)>0$ for a.e. $E\in\R$. Moreover, for small $\lambda$ and smooth non-constant $v$ one has $L(E)\approx \la^2$ for $\sqrt{\la}<|E|<2-\sqrt{\la}$ \cite{cs}; 
and for large $\lambda$, and under quite general conditions on $v$, one has $L(E)\gtrsim \log\lambda$ for all $E$ outside an exponentially small (in $\la$) set \cite{shs,s}. 
 
Furthermore, from Herman's subharmonic argument \cite{h} it follows that if $v$ is a non-constant trigonometric polynomial, 
then $L(E)\gtrsim \log\lambda$ for all $E\in \R$ and all large $\lambda$. However, whether the corresponding result holds for $v$ a non-constant real-analytic
function does not seem to be known (if instead $T(x)=x+\omega$, $\omega\in\mathbb{R}\setminus \mathbb{Q}$, this is a well-known result \cite{ss}). 
Steps towards a proof of this result are taken in \cite{m}.
Another situation where one has $L(E)\gtrsim \log\lambda$ for all $E\in \R$ and all large $\lambda$ is when $v$ is $C^1$ and monotone on $(0,1)$, 
with a discontinuity at $x=0$ \cite{z}.

On the other hand, if $\varphi:\T\to\mathbb{R}$ is any bounded and measurable function such that $\int_{\T}\varphi(x)dx=0$, and 
we let $v(x)=\exp(\varphi(T(x)))+\exp(-\varphi(x))$ and take $\la=1$, then $L(0)=0$ (see \cite{b2}). Thus there are obstacles for obtaining uniform (in $E$) lower bounds
on $L(E)$.

The aim of this paper is to extend the above results to new situations. We first define the collection of potential functions $v$  with which we shall work.
\begin{defi}\label{def_v}
Let $\mathcal{V}^1(\T,\R)$ denote the class of $C^1$-functions $v:\T\to\mathbb{R}$ which satisfy the following condition: there exist $\ve_0>0, \beta>0$ and an integer $s\geq 1$ 
such that for all $0<\ve\leq \ve_0$ and all $a\in\R$ the set $\{x\in\T: |v(x)-a|<\ve\}$ consists of at most $s$ intervals, each of length at most 
$\ve^{\beta}$.
\end{defi}
\begin{rem}It is easy to check that all non-constant real-analytic functions $v:\T\to\R$ belongs to  $\mathcal{V}^1(\T,\R)$.
Note also that the assumption on $v$ is very similar to \cite[Definition 2.2]{k}.
\end{rem}

Our main result is the following:
\begin{thm}
Assume that $v\in\mathcal{V}^1(\T,\R)$. Then there is a $\lambda_0=\lambda_0(v)>0$ such that for all $\lambda\geq \lambda_0$ we have
$L(E)>(\log\la)/4$ for all $E\in \R$ provided that $b\geq \la^3$.

\end{thm}
\begin{rem}(a) We do not aim for optimal conditions on any of the constants or required size of $b$ (as function of $\la$). 

(b) It would be very interesting to know if the statement of Theorem 1 holds true for a fixed (large) $b$ independent of $\lambda$. Unfortunately the method
we use in the proof requires $b$ to be (much) larger than $\lambda$. (Heuristically it should be more likely to have $L(E)>0$ the bigger is $\lambda$.)

(c) We would like to stress that the proof of Theorem 1 does not use the fact that we have a linear system. In fact one can extend it to non-linear
systems (as, e.g., we did in \cite{b1}) since what we analyze is the dynamics of forced circle diffeomorphisms (see the map (\ref{circle_diff})). However, since
there is an interest in the Schr\"odinger cocycle, and for (hopefully!) transparency, we perform our analysis for this explicit system.

(d) The proof is based on ideas developed in \cite{v}. Related problems (for systems which are not homotopic to the identity) are investigated in \cite{b1,y}. 
\end{rem}

\section{Preliminaries}

We adopt the following convention on the coupling constant $\lambda$: In the statements of the lemmas below (where applicable) we always assume (without explicitly stating so) that 
$\lambda>0$ is sufficiently large. 
There are only finitely many largeness conditions on $\la$, and they only depend on $v$. This will yield the constant $\la_0$ in the statement of Theorem 1. 

\subsection{Assumption on $v$} We assume from now on that $v\in\mathcal{V}(\T,\R)$ is fixed, and that $\ve_0,\beta$ and $s$ are as in Definition \ref{def_v}.
 Without loss of generality we assume, for simplicity, that $|v(x)|\leq 1/3$ for all $x\in\T$ (this only scales $\la$).

\subsection{Projective action} Since
$$
A_E(x)\binom{1}{r}=r\binom{1}{\la v(x)-E-1/r}
$$
we see that the cocycle $F_E$ induces an action on the projective space $\mathbb{P}^1(\R^2)$ (with coordinates $\binom{1}{r}$) given by
$$
G_E(x,r)=(T(x),\la v(x)-E-1/r).
$$

If $r_n=\pi_2(G_E^n(x,r)), n\geq 0,$ it is easy to verify that\footnote{Here $\pi_2$ denotes the projection $\pi_2(x,r)=r$.}
\begin{equation}\label{norm_prod}
A^{n+1}_E(x)\binom{1}{r}=\binom{r_{n}\cdots r}{r_{n+1}\cdots r}.
\end{equation}
Thus, if there for some parameter $E\in\R$ exists a set $X\subset \T$ of positive measure such that for each $x\in X$ there is an $r\in\R$ such that  
$$
\limsup_{n\to\infty}\frac{1}{n}\log\|r_n\cdots r\|\geq C,
$$
then it follows from (\ref{pointwise}) and (\ref{norm_prod}) that $L(E)\geq C$.   

\subsection{Bounds} 
\begin{lemma}\label{L_bounds0}
If $|E|\geq \la/3+2\sqrt{\la}$, then $L(E)\geq (\log\la)/2$.
\end{lemma}
\begin{proof}Take any $x\in \T$ and $|r|>\sqrt{\la}$, and define $r_n=\pi_2(G_E^n(x,r)), n\geq 1$.
We note that $|r_1|=|\la v(x)-E-1/r|\geq|E|-|\la v(x)|-|1/r|>\sqrt{\la}$. By induction we get
$|r\cdot r_1\cdots r_n|>\sqrt{\la}^{n+1}$ for all $n\geq 1$.
\end{proof}

Thus, in order to prove Theorem 1, we only need to consider $|E|<\la/3+2\sqrt{\la}$ (which of course is the cumbersome region).

\begin{lemma}\label{L_bounds1}
Assume that $|E|\leq \la/3+2\sqrt{\la}$. For any $N\geq 1$, if $x\in\T$ and $r\in\R\cup\{\infty\}$ are such that 
$|r_j|<\sqrt{\la}$ for at most $k$ indices $j$ in $[1,N]$, then
$$
|r_1\cdots r_N|\geq \sqrt{\la}^{N-3k} \text{ if } |r_N|\geq 1/\la; \quad |r_1\cdots r_{N+1}|\geq \sqrt{\la}^{N+1-3k} \text{ if } |r_N|< 1/\la.
$$ 
\end{lemma}
\begin{proof}The statement follows easily by induction over $N$. Note the following: If $|r_j|<1/\la$, then
$|r_jr_{j+1}|=|r_j(\la v(x_j)-E)-1|>1/4$.
\end{proof}

\subsection{Elementary probability}\label{subs_prob} 
We begin by defining some natural partitions of $\T$ (relative the transformation $T$). First, let
$$
I_j=\left[\frac{j-1}{b},\frac{j}{b}\right), \quad j=1,2,\ldots, b.
$$
Note that $I_1\cup I_2\cup\cdots \cup I_b=\T$ and $T(I_j)=\T$ for all $j$. 

Next we define the intervals 
$$
I_{j_1j_2\cdots j_n}=I_{j_1}\cap T^{-1}(I_{j_2})\cap\ldots \cap T^{-n+1}(I_{j_n}), \quad n\geq 2.
$$
Thus, by definition, if $x\in I_{j_1j_2\cdots j_n}$, then $x\in I_{j_1}$, $T(x)\in I_{j_2}, \ldots, T^{n-1}(x)\in I_{j_n}$; and 
clearly $I_{j_1\ldots j_n}\subset I_{j_1\ldots j_{n-1}}\subset \ldots\subset I_{j_1}$.
Note that the length of each interval $I_{j_1j_2\cdots j_n}$ is $b^{-n}$. 
Note also that for any  fixed interval $I_{j_1j_2\cdots j_n}$ the intervals $I_{j_1j_2\cdots j_nj}$, $1\leq j\leq b$, form a partition 
of $I_{j_1j_2\cdots j_n}$ into $b$ pieces of equal length $b^{-{(n+1)}}$.

In the following lemma we use the word "bad" to indicate that an interval does not have a certain property:
\begin{lemma} Let $1\leq q\leq b$ be an integer. Assume that $q$ of the intervals $I_j$ ($j=1,\ldots, b$) are bad.
Furthermore, for each interval $I_{j_1\ldots j_n}$ assume that $q$ of the intervals $I_{j_1\ldots j_n j}$ ($j=1,2,\ldots, b$)
are bad. Then for each $n\geq 1$ and $0\leq m\leq n$ the set 
$$
\{x\in\T: x\in I_{j_1\ldots j_n}  \text{ and exactly } m \text{ of the intervals } I_{j_1}, I_{j_1j_2}, \ldots, I_{j_1\ldots j_n} 
\text{ are bad}\}
$$
has measure $\binom{n}{m}\frac{q^m(b-q)^{n-m}}{b^n}$. 
\end{lemma}
\begin{proof}Easy combinatorics. From the assumption it follows that there are $\binom{n}{m}q^m(b-q)^{n-m}$ $n$-tuples $(j_1,\ldots,j_n)\in\{1,\ldots,b\}^n$ such that
exactly $m$ of the intervals  $I_{j_1}, I_{j_1j_2}, \ldots, I_{j_1\ldots j_n}$ are bad. Using the fact that $|I_{j_1\ldots j_n}|=b^{-n}$ yields the result. 
\end{proof}

Moreover,
\begin{lemma}\label{L_prob1} 
With the same assumptions as in the previous lemma, the measure of the set
$$
M_n=\{x\in\T: x\in I_{j_1\ldots j_n}  \text{ and at most } [2(q/b)n] \text{ of the intervals } I_{j_1}, I_{j_1j_2}, \ldots, I_{j_1\ldots j_n} 
\text{ are bad}\}
$$
goes to $1$ as $n\to\infty$.
\end{lemma}
\begin{proof}By the previous lemma we have (provided that $[2(q/b)n]\leq n$)
$$
|M_n|=\sum_{m=0}^{[2(q/b)n]}\binom{n}{m}\frac{q^m(b-q)^{n-m}}{b^n}=\sum_{m=0}^{[2(q/b)n]}\binom{n}{m}(q/b)^m(1-q/b)^{n-m}.
$$
Applying the de Moivre-Laplace theorem yields the result.
\end{proof}

\section{Geometry} 
We write the circle $S^1=[-\pi/2,\pi/2]/\sim$. A point $\binom{1}{r}$ in projective coordinates corresponds to $\arctan(r)\in S^1$.
Thus, in these coordinates the map $G_E$ becomes $H_E:\T\times S^1\to \T\times S^1$ given by 
\begin{equation}\label{circle_diff}
H_E(x,y)=(T(x),\arctan(\la v(x)-E-1/\tan(y))).
\end{equation}

The following lemma is crucial for the proof of Theorem 1 (compare with "admissible curves" in \cite{v}; and also the idea in \cite{v} that the image of an admissible curve
"spreads out" in the $y$-direction). Recall that the constants $\beta$ and $s$ come from the assumption on $v$.

\begin{lemma}\label{L_geom}
Assume that $|E|<\la/2$ and $b\geq \la^3$. Assume further that the function $\vf:[0,1]\to S^1$ is $C^1$ and satisfies 
$$|\vf'(x)|<\frac{\la\|v'\|}{b}(1+(2\la^2)/b+\ldots +(2\la^2)^m/b^m)$$ for all $x\in[0,1]$ and some integer $m\geq 0$. 
For $j=1,2,\ldots,b$, let $\vf^j:[0,1]\to S^1$ be defined by
$$\vf^j(x)=\arctan\left(\la v\left(\frac{x+j-1}{b}\right)-E-\frac{1}{\tan\left(\vf\left(\frac{x+j-1}{b}\right)\right)}\right).$$
Then the following hold:
\begin{enumerate}

\item There are at most $(s+1)(2+[2^\beta\la^{-\beta/2}b])$ indices $j\in\{1,\ldots,b\}$ for which $$\min_{x\in [0,1]}|\tan\vf^j(x)|<\sqrt{\la}.$$

\item The estimate $|(\vf^j)'(x)|<\frac{\la\|v'\|}{b}(1+(2\la^2)/b+\ldots +(2\la^2)^{m+1}/b^{m+1})$ holds for all $x\in[0,1]$ and all $j$. 
\end{enumerate}
\end{lemma}
\begin{rem}That the functions $\vf^j$ are defined as a above implies that if $\Gamma_j=\{(x,\vf(x)): x\in I_j\}$, then 
$H_E(\Gamma_j)=\{(x,\vf^j(x)):x\in [0,1)\}$. Thus, if $\Gamma=\{(x,\vf(x)): x\in [0,1)\}$ is the graph of $\vf$, we have 
$H_E(\Gamma)=\bigcup_{j=1}^b \{(x,\vf^j(x)):x\in [0,1)\}$, i.e., the union of the graphs of the $\vf^j$.

The statement thus says that we have a bound of the number of indices  $j$ for which the graph $\{(x,\vf^j(x)):x\in [0,1)\}$ intersects the ``bad region'' 
$[0,1]\times (-\arctan\sqrt{\la},\arctan\sqrt{\la})$;
and the derivative estimate shows that we can iterate this process (iterate each of the $j$ graphs) and still have a good control on the derivative 
(provided that $b$ is large enough, independently of $m$).

\end{rem}

\begin{proof} By the assumptions on $b$ and $\vf'$ we have $|\vf'|<(\la\|v'\|/b)\sum_{i=0}^\infty (2\la^2/b)^i<2\|v'\|/\la^2$. We let $g(x)=-1/\tan(\vf(x))$.

We first prove that the set $B:=\{x\in[0,1]:|\la v(x)-E+g(x)|<\sqrt{\la}\}$ 
can intersect at most $(s+1)(2+[2^\beta\la^{-\beta/2}b])$ of the intervals $I_j$ ($j=1,\ldots, b$). This clearly gives the first statement of the lemma.

If $|\vf(x_0)|<1/\la$ for some $x_0$, 
then the estimate on $|\vf'|$ implies that
$|\vf(x)|<1/\la+ 2\|v'\|/\la^2$ for all $x$; thus, since $|\la v(x)-E|\leq \la/3+\la/2$, we get $|\la v(x)-E+g(x)|>\sqrt{\la}$ for all $x\in [0,1]$. 
We conclude that  $B=\emptyset$ in this case. 

Assume now that $|\vf(x)|\geq 1/\la$ for all $x$.
Then we have $|g'(x)|=|\vf'(x)|/|\sin^2(\vf(x))|<(2\|v'\|/\la^2)(2\la^2)=4\|v'\|$, and therefore $|g(x_1)-g(x_0)|<4\|v'\|$ for all $x_0,x_1\in[0,1]$
Hence  $$B\subset \{x\in[0,1]:|\la v(x)-E+g(0)|<2\sqrt{\la}\}.$$ By the assumption on $v$
the latter set consists of at most $s+1$ intervals ($s$ intervals on $\T$ can be at most $s+1$ intervals in $[0,1]$), 
each of a length $<(2/\sqrt{\la})^\beta$. Since each interval $I_j$ ($j=1,\ldots, b$)
has length $1/b$, it follows that the set $B$ can intersect at most $(s+1)(2+[2^\beta\la^{-\beta/2}b])$ of them.

We turn to the derivative estimate in (2). From the assumptions on $v$ and $E$ we have $|\la v(t)-E|<\la$ for all $t$. An easy computation shows that
$|(\vf^j)'(x)|\leq (\la\|v'\|+\|\vf'\|2\la^2)/b$, from which the desired bound follows.
To obtain the estimate in the second term we have used the fact that if $|a|\leq \la$, and $\la$ is sufficiently large (larger than a numerical constant),  then 
$\sin^2t+(a\sin t-\cos t)^2>1/(2\la^2)$ for all $t$. 

\end{proof}

\section{Proof of Theorem 1}
By Lemma \ref{L_bounds0} we only need to consider $|E|<\la/2$. 
We therefore assume that $|E|<\la/2$ is fixed. We also fix $b\geq \la^3$ (so that we can apply Lemma \ref{L_geom}).
Given a point $(x,r)$ we denote by $r_j$ the iterate $r_j=\pi_2(G_E^j(x,r))$.

Let the intervals $I_{j_1\ldots j_n}$ be defined as in subsection \ref{subs_prob}. We say that the interval $I_{j_1\ldots j_n}$ ($n\geq 1$)
is "good" if for each $x\in I_{j_1\ldots j_n}$ and $r=\la$ we have $|r_n|\geq \sqrt\la$; otherwise the interval is "bad".

\begin{lemma}Assume that at most $q=[b/12]$ of the intervals $I_j$ $(j=1,\ldots, b)$ are bad, and that for each interval
$I_{j_1\ldots j_n}$ at most $q$ of the intervals $I_{j_1\ldots j_nj}$ $(j=1,\ldots, b)$ are bad. Then $L(E)\geq (\log\la)/4$.
\end{lemma}
\begin{proof}
Let the sets $M_n$ be defined as in Lemma \ref{L_prob1}, and let $\displaystyle M=\limsup_{n\to\infty} M_n$. Since $|M_n|\to 1$ as $n\to\infty$ we have that
$|M|=1$, i.e., the set $M$ has full measure.   

Take $x\in M_n$ ($n\geq 1$) and $r=\la$. Then $x\in I_{j_1\ldots j_n}\subset I_{j_1\ldots j_{n-1}}\subset \ldots\subset I_{j_1}$ and at most 
$[2(q/b)n]$ of the intervals $I_{j_1},\dots, I_{j_1\ldots j_n}$ are bad. Thus, by definition we get 
that $|r_k|<\sqrt{\la}$ for at most $[2(q/b)n]\leq n/6$ indices $k\in[1,n]$.
It hence follows from Lemma \ref{L_bounds1} that 
$|r_1\cdots r_n|\geq \sqrt{\la}^{n-3n/6}$ or $|r_1\cdots r_{n+1}|\geq \sqrt{\la}^{n+1-3n/6}$. 

Consequently, if $x\in M$ (and thus $x\in M_n$ for infinitely many $n$) and $r=\la$
we have 
$$
\limsup_{n\to\infty}\frac{1}{n}\log|r_1\cdots r_n|\geq (1/4)\log\la. 
$$
Recalling (\ref{pointwise}) and (\ref{norm_prod}) finishes the proof.
\end{proof}

Combining the previous lemma with the next one finishes the proof of Theorem 1.
\begin{lemma}At most $q=[b/12]$ of the intervals $I_j$ ($j=1,\ldots, b$) are bad; and for each interval
$I_{j_1\ldots j_n}$ at most $q$ of the intervals $I_{j_1\ldots j_nj}$ ($j=1,\ldots, b$) are bad.
\end{lemma}
\begin{proof}
The strategy is to apply Lemma \ref{L_geom} inductively, and we shall begin by iterating the constant graph $\{(x,\arctan \la):x\in [0,1]\}$.

Note that $(s+1)(2+[2^\beta\la^{-\beta/2}b])<[b/12]$ if $\lambda$ is sufficiently large (recall that $b\geq \la^3$), where
the left hand side is the quantity in  Lemma \ref{L_geom}(1). Since $E$ is fixed we write $H=H_E$ and $G=G_E$.

For $j=1,2,\ldots, b$, let $\vf_j:[0,1]\to S^1$ be defined by
$$
\vf_j(x)=\pi_2\left(H((x+j-1)/b,\arctan{\la})\right).
$$
Applying Lemma \ref{L_geom} with $\varphi(x)=\arctan(\la)$ and $m=0$ shows that we for each $j$ have $|\vf_j'(x)|<\frac{\la\|v'\|}{b}(1+(2\la^2)/b)$
and that there are at most $[b/12]$ indices $j\in\{1,\ldots, b\}$ for which $\min_{x\in[0,1]}|\tan\vf_j(x)|<\sqrt{\la}$.
We note that 
$
H(I_{j}\times\{\arctan\la\})=\{(x,\vf_{j}(x)):x\in [0,1)\}, 
$
and thus $$G(I_{j}\times\{\la\})=\{(x,\tan(\vf_{j}(x))):x\in [0,1)\}.$$ Consequently, at most $[b/12]$ of the intervals $I_j$ are bad.
This proves the first statement of the lemma.

Inductively (over $n\geq 1$) we define, for fixed $(j_1,\ldots,j_n)\in\{1,\ldots, b\}^n$, the functions $\vf_{j_1\ldots j_n j}:[0,1]\to S^1$ by
$$
\vf_{j_1\ldots j_nj}(x)=\pi_2\left(H((x+j-1)/b,\vf_{j_1\ldots j_n}((x+j-1)/b))\right), \quad j=1,2,\ldots, b.
$$
Lemma \ref{L_geom} gives $|\vf_{j_1\ldots j_nj}'(x)|<\frac{\la\|v'\|}{b}(1+(2\la^2)/b+\ldots + (2\la^2)^{n+1}/b^{n+1})$ and that
there are at most $[b/12]$ indices $j\in\{1,\ldots, b\}$ for which $\min_{x\in[0,1]}|\tan\vf_{j_1\ldots j_nj}(x)|<\sqrt{\la}$.

From the above definition it is easy to verify that for each $n\geq 1$ and each fixed $(j_1,\ldots,j_n)\in \{1,\ldots, b\}^n$ we have
$$
\begin{aligned}
H^{n+1}(I_{j_1\ldots j_nj}\times\{\arctan\la\})&=\{(x,\vf_{j_1\ldots j_nj}(x)):x\in [0,1)\}; \text{ and } \\
G^{n+1}(I_{j_1\ldots j_nj}\times\{\la\})&=\{(x,\tan\vf_{j_1\ldots j_nj}(x)):x\in [0,1)\}.
\end{aligned}
$$
Thus, at most $[b/12]$ of the intervals $I_{j_1\ldots j_nj}$ ($1\leq j\leq b$) are bad. This finishes the proof of the lemma.

\end{proof}

\end{document}